\documentclass[a4paper]{article}

\usepackage{amsfonts}
\usepackage{amssymb,amsthm,amsmath,graphicx,bm,ebezier,amsthm}
\usepackage{hyperref}
\usepackage{caption}
\usepackage[normalem]{ulem}
\usepackage{url,lineno}
\usepackage[utf8]{inputenc}
\usepackage{arydshln}
\usepackage[ruled,vlined,linesnumbered]{algorithm2e}
\usepackage[noend]{algpseudocode}
\usepackage{lscape}
\usepackage{multirow}
\usepackage[dvipsnames]{xcolor}
\usepackage{stmaryrd}

\usepackage{todonotes}

\newtheorem{theorem}{Theorem}
\newtheorem{proposition}[theorem]{Proposition}
\newtheorem{corollary}[theorem]{Corollary}
\newtheorem{lemma}[theorem]{Lemma}

\newtheorem{example}[theorem]{Example}
\newtheorem{remark}[theorem]{Remark}

\def\CaH{\mathcal{H}}
\def\CaC{\mathcal{C}}
\def\CaA{\mathcal{A}}

\def\msg{\operatorname{msg}}
\def\rank{\operatorname{rank}}
\def\g{\operatorname{g}}
\def\m{\operatorname{m}}
\def\r{\operatorname{r}}
\def\e{\operatorname{e}}
\def\N{\operatorname{N}}

\def\Fb{\operatorname{Fb}}
\def\SG{\operatorname{SG}}
\def\PF{\operatorname{PF}}
\def\Ap{\operatorname{Ap}}
\def\M{\operatorname{M}}
\def\AMED{\operatorname{\CaA MED}}
\def\MED{\operatorname{MED}}

\title{Affine semigroups without consecutive small elements}
\date{}
\author{J. C. Rosales, R. Tapia-Ramos and A. Vigneron-Tenorio}

\begin{document}
\maketitle

\abstract{An $\CaA$-semigroup is a numerical semigroup without consecutive small elements. This work generalizes this concept to finite-complement submonoids of an affine cone $\CaC$. We develop algorithmic procedures to compute all $\CaA$-semigroups with a given Frobenius element (denoted by $\CaA(f)$), and with fixed Frobenius element and multiplicity. Moreover, we analyze the $\CaA(f)$-systems of generators.
Furthermore, we study $\CaA$-numerical semigroups with maximal embedding dimension, fixed Frobenius number and multiplicity, providing an algorithm for their computation and a graphical classification.
}

{\small

}

\section*{Introduction}

Let $\mathbb N$ be the set of all non-negative integers, and consider $[k]$ as the set $\{1,\ldots ,k\}$. Given $p\in \mathbb N\setminus \{0\}$, an affine semigroup $S$ in $\mathbb N^p$ is a non-finite subset of $\mathbb N^p$ closed under addition, containing the zero element, and such that there exists a finite subset $A=\{a_1,\ldots ,a_k\}\subset \mathbb N^p$ with $S=\{\sum_{i\in [k]}\lambda_i a_i\mid \lambda_1,\ldots ,\lambda_k\in \mathbb N \}$,
this set is denoted by $\langle A \rangle$, and it is named a generating set of $S$.
If this set is minimal with respect to inclusion, it is the unique minimal generating set of $S$ and is denoted by $\msg(S)$. The minimal generating set exists for any affine semigroup (see \cite{rosales1999finitely}), and its cardinality, called the embedding dimension of $S$, is denoted by $\e(S)$.

Given $\mathcal C\subseteq \mathbb N^p$ a non-negative integer cone, an affine semigroup is a $\mathcal C$-semigroup if the minimal integer cone containing $S$ is $\mathcal C$, and $\CaH(S)=\mathcal C\setminus S$ is a finite set. We assume that the cone has at least $p$ extremal rays. The canonical basis of $\mathbb N^p$  is $\Lambda=\{e_1,\ldots ,e_p\}$. In case $p=1$, then $S$ is a numerical semigroup.

Several invariants of $\CaC$-semigroups are needed for this work. Consider a $\CaC$-semigroup $S$, and a fixed monomial order $\preceq$ on $\mathbb N^p$, that is, a total order on $\mathbb N^p$ compatible with addition, where $0 \preceq x$ for any $x\in \mathbb N^p$ (see \cite{cox1997ideals}), the Frobenius element of $S$ is defined as $\Fb(S)=\max_\preceq \CaH(S)$. When $\CaH(S)$ is empty, $\Fb(S)=-\sum_{i\in [p]} e_i$.  Any $x\in S$ is called a small element of $S$ if $x\prec \Fb(S)$, and the set of all small elements of $S$ is denoted by $\N(S)$. The multiplicity of $S$ is $\m(S)=\min_\preceq (S\setminus\{0\})$, and its ratio $\r(S)=\min_\preceq \big(S\setminus\langle \m(S) \rangle\big)$.

The concept of $\CaA$-numerical semigroup was introduced in \cite{rosales2023numerical}: a numerical semigroup $S$ is called an $\mathcal A$-semigroup if, for any $x\in \mathbb N$ where $x<f$, the condition $\{x,x+1\}\not\subset S$ holds. This work expands this definition to $\CaC$-semigroup as: a $\CaC$-semigroup $S\subseteq\CaC$ is an $\CaA$-semigroup when, for every $s\in S$ and $e\in \Lambda$, the set $\{s,s+e\}$ is not contained in $\N(S)$. Given $x,y\in \mathbb N^p$, we say that $x$ and $y$ are consecutive if $y-x\in \Lambda$, or $x-y\in \Lambda$. Thus, an $\CaA$-semigroup is a $\CaC$-semigroup without consecutive small elements. Fixed a cone $\CaC$, the set of all $\CaC$-semigroups with Frobenius element $f$ is denoted by $\CaA(f)$.

One of the main objectives of this work is to study the fundamental properties of the set $\CaA(f)$. We prove that $\CaA(f)$ is a covariety, and we use it to provide some algorithmic methods to compute it. In particular, fixed $f\in \CaC$, a finite tree containing all the $\CaA$-semigroups with Frobenius $f$ is introduced. These $\CaA$-semigroups are obtained by joining some special gaps to a given $\CaA$-semigroup. To improve this computation, we propose an approach to get these special gaps from their father's special gaps.

A partition of $\CaA(f)$ is defined via the subsets $\CaA(f,m) = \{S \in \CaA(f) \mid \m(S) = m\}$. It is shown that $\CaA(f)$ is the finite union of such subsets, where each value of $m$ satisfies specific properties. Since the set $\CaA(f,m)$ is a ratio-covariety, we construct an associated tree and provide an algorithm to compute it. Furthermore, we introduce an additional partition of $\CaA(f, m)$ based on certain subsets of $\CaC$ determined by the parameters $f$, $m$, and the fixed monomial order. The above different partitions allow us to develop other algorithms for computing $\CaA(f)$.

A set $X\subset \CaC$ is an $\CaA(f)$-set if there is a $S\in \CaA(f)$ such that $X\subseteq \N(S)\setminus\{0\}$. We show that given $X$ an $\CaA(f)$-set, there exists the smallest element $S$ in $\CaA(f)$ containing $X$. This $\CaA$-semigroup is denoted by $\CaA(f)[X]$ and $X$ is called an $\CaA(f)$-system of generators of $S$. We prove that every semigroup in $\CaA(f)$ admits a unique minimal $\CaA(f)$-system of generators. The cardinality of this minimal $\CaA(f)$-system of generators is known as the $\CaA(f)$-rank of $S$. The $\CaA$-semigroups with $\CaA(f)$-rank equal to zero, one, or two are also characterized in this work.

Focus on numerical semigroups; it is well-known that the embedding dimension of a numerical semigroup is upper-bounded by $\m(S)$. The numerical semigroups whose embedding dimension is equal to their multiplicity are called semigroups with maximal embedding dimension, and are denoted as $\MED$-semigroups. In connection with this class of semigroups, a long list of publications exists where $1$-dimensional local analytically irreducible domains are studied via their value semigroups. From this approximation, several properties of such rings have been considered, including being an $\MED$-semigroup, Arf numerical semigroup or saturated (see \cite{Barucci}, \cite{Lipman}, \cite{Sally}, \cite{zariski1&2}, and \cite{zariski3}). Based on the above relationship, concepts such as the numerical $\MED$-semigroup, Arf numerical semigroup, and saturated numerical semigroup were born. Following the terminology introduced in \cite{rosales2023numerical}, an $\AMED$-semigroup is a numerical $\MED$-semigroup which is also $\CaA$-semigroup. Denoting by $\AMED (f,m)$ the set of numerical $\AMED$-semigroup with Frobenius number $f$ and multiplicity $m$, we prove that $\AMED (f,m)$ is a ratio-covariety, and we provide an algorithm for computing it. We also prove that the set of Arf numerical semigroups with  Frobenius number $f$ ($\operatorname{Arf}(f)$), and saturated numerical semigroups with Frobenius number $f$ ($\operatorname{Sat}(f)$) are covarieties, and $\operatorname{Sat}(f)\subseteq \operatorname{Arf}(f)\subseteq \CaA(f)$.

This work is organized into five sections: Section \ref{Sec1} introduces a tree of $\CaA$-semigroups with a fixed Frobenius element, and discusses some of their properties. In Section \ref{Sec2}, we study the minimal $\CaA(f)$-system of generators, and $\CaA$-semigroups with $\CaA(f)$-rank equal to zero, one, or two. Sections \ref{Sec3} and \ref{Sec4} describe some partitions of the sets $\CaA(f)$ and $\CaA(f,m)$, providing algorithms to compute them. Finally, Section \ref{Sec5} focuses on numerical $\MED$-semigroups, Arf numerical semigroups, and saturated numerical semigroups, which are also $\CaA$-semigroups. The results of this work are illustrated with several examples.

\section{The set $\CaA(f)$ and its associated tree}\label{Sec1}

Let $\Lambda=\{e_1,\ldots, e_p\}$ be the set of canonical basis vectors of $\mathbb{N}^p$, and $\preceq$ be a monomial order on $\mathbb{N}^p$, recall that an $\CaA$-semigroup $S\subseteq\mathbb{N}^p$ is a $\CaC$-semigroup such that for every $s\in S$ and every $e\in \Lambda$, if $s+e\in S$, then $s+e\notin \N(S)$. Also, given $f \in \CaC \setminus \{0\}$,
\[
\CaA(f)=\{S \text{ is a } \CaA\text{-semigroup}\mid \Fb(S)=f\}.
\]
This section analyses the set $\CaA(f)$ by classifying it graphically into a tree, which leads to an algorithm for computing $\CaA(f)$.

Since $\N(S)$ depends on the monomial order $\preceq$ choice, the set $\CaA(f)$ also does, as shown in the following example.

\begin{example}
Consider the positive cone $\CaC$ spanned by $\{(1,0), (1,1), (1,2)\}$ and the $\CaC$-semigroup $S$ whose set of gaps is
\[
\CaH(S)=\{(1,0), (1,1), (2,0), (2,1), (2,2), (3,1), (4,0), (4,1), (4,2), (5,0)\}.
\]
Let $\preceq_{glex}$ be the degree lexicographical order, and $\preceq_{M}$ be the matrix order defined by the matrix $M=\left(\begin{array}{cc} 1 & 1\\ 2 & 0\end{array}\right)$. Note that for any $a,b\in \mathbb N^2$, $a\preceq_{M} b$ if $aM\preceq_{lex} bM$ (see \cite{cox1997ideals}). We obtain that the set of small elements of $S$ is
$$\N_{\preceq_{glex}}(S)=\{(0,0), (1,2), (2,3), (3,0), (3,2), (5,1), (6,0)\},$$
with respect to $\preceq_{glex}$,
and $\N_{\preceq_{M}}(S)=\N_{\preceq_{glex}}(S)\cup\{(7,0)\}$ respect to $\preceq_{M}$.
Thus, $S$ is an $\CaA$-semigroup for $\preceq_{glex}$, but it is not an $\CaA$-semigroup for $\preceq_{M}$, since $(6,0) + (1,0) = (7,0) \in \N_{\preceq_{M}}(S)$.
\end{example}

In accordance with the terminology established in \cite{moreno2024covariety}, a covariety $V$ is a non-empty family of $\CaC$-semigroups satisfying the following conditions:
\begin{itemize}
\item[(i)] There exists a minimal element in $V$ with respect to inclusion, denoted by $\min_{\subseteq}V$.
\item[(ii)] If $S,T\in V$, then $S\cap T\in V$.
\item[(iii)] If $S\in V\setminus\{\min_{\subseteq}V\}$, then $S\setminus\{\m(S)\}\in V$.
\end{itemize}
Let $f\in \CaC\setminus\{0\}$. Notice that,
\[
\Delta(f)=\operatorname{min}_{\subseteq}\bigl(\CaA(f)\bigr)=\{0\} \cup\{x\in\CaC \mid x\succ f\}.
\]
In the context of numerical semigroups, $\Delta(f)$ is known as a half-line semigroup (see \cite{libroRosales}). At the same time, in a more general setting, it is referred to as an ordinary semigroup (see \cite{rosales2025computational}).

\begin{proposition}\label{A(f)covariedad}
Let $f\in\CaC\setminus\{0\}$. Then, the set $\CaA(f)$ is a covariety.
\end{proposition}

\begin{proof}
It suffices to check the last two properties to show that $\CaA(f)$ is a covariety.
The intersection of two $\CaC$-semigroups $S$ and $T$ is itself a $\CaC$-semigroup, and its Frobenius element corresponds to $\max_\preceq\{\Fb(S), \Fb(T)\}$. Moreover, the property of being an $\CaA$-semigroup is compatible with the intersection.
Finally, for any $S\in \CaA(f)\setminus\{\Delta(f)\}$, we obtain that $S \setminus \{\m(S)\}$ is also a $\CaC$-semigroup, and since $S$ is an $\CaA$-semigroup, removing $\m(S)$ preserves the $\CaA$-semigroup structure. Thus, $\CaA(f)$ is a covariety.
\end{proof}

The properties of covarieties allow us to describe the set $\CaA(f)$ through an associated directed graph $G(\CaA(f))$, whose vertex set is $\CaA(f)$, and where an ordered pair $(S,T)\in \CaA(f)$ with $S\ne T$ forms an edge if and only if $T = S \setminus \{\m(S)\}$. From now on, if $(S,T)$ is an edge, then we say that $S$ is a child of $T$.

A graph is a tree if there exists a vertex $R$ (the root) such that any vertex $S$ is connected to $R$ by a unique sequence of edges \[(S_0, S_1), (S_1, S_2), \ldots, (S_{n-1}, S_n),\]
where $S_0 = S$, and $S_n = R$.

We establish the following result by extending the scope of \cite[Proposition 2.3]{moreno2024covariety}.

\begin{proposition}\label{prop:grafoA_f}
The graph $G(\CaA(f))$ is a tree with root $\Delta(f)$.
\end{proposition}

\begin{proof}
Let $S\in \CaA(f)\setminus\{\Delta(f)\}$, and consider the sequence $\{S_i\}_{i\in\mathbb{N}}$ given by $S_0 = S$ and, for each $i\in \mathbb{N}$, $S_{i+1} = S_{i} \setminus \{\m(S_{i})\}$ if $S_{i} \ne \Delta(f)$, and $S_{i+1} = \Delta(f)$ otherwise.
Since the sequence is stationary, it defines a sequence of edges that connect $S$ to $\Delta(f)$. The uniqueness of this sequence follows from the fact that each $\CaC$-semigroup has a unique multiplicity, ensuring no cycles occur.
\end{proof}

For any $\CaC$-semigroup $S$, we say that an element $x \in \CaH(S)$ is a pseudo-Frobenius element if $x + s \in S$ for all $s \in S \setminus \{0\}$. The set of all pseudo-Frobenius elements of $S$ is denoted by $\PF(S)$. A pseudo-Frobenius element $x \in \PF(S)$ is a special gap if $2x \in S$, and the set of all special gaps is denoted by $\SG(S)$. One of the main properties of this invariant, justifying its introduction, is that $S \cup \{x\}$ is a $\CaC$-semigroup if and only if $x$ is a special gap of $S$.

As an immediate consequence of the above proposition, and using \cite[Proposition 2.4]{moreno2024covariety}, we obtain the following result, which provides a characterization of the set of children of a given $\CaC$-semigroup in $G(\mathcal{A}(f))$.

\begin{proposition}
Let $f\in \CaC\setminus\{0\}$ and $T \in \CaA(f)$. The set of children of $T$ in the tree $G(\CaA(f))$ is the set,
\[\bigl\{ T \cup \{x\} \mid x \in \SG(T)\setminus(\Lambda \cup \{f\}),\, x \prec \m(T), \text{ and } x +e \notin \N(T) \text{ for all } e\in \Lambda \bigr\}.\]
\end{proposition}

\begin{proof}
Suppose that $S$ is a child of $T$, which implies that $T \cup \{x\} = S$, where $x = \m(S)$, and thus $x \in \SG(T)\setminus(\Lambda \cup \{f\})$. From the definition of $\m(S)$, it follows that $x \prec \m(T)$. Since $T$ and $S$ belong to $\CaA(f)$, then $ x +e \notin \N(S)=\N(T)\cup\{x\}$ for any $e\in \Lambda$.

Conversely, suppose that $S = T \cup \{x\}$, where $x \in \SG(T)\setminus(\Lambda \cup \{f\})$, $x \prec \m(T)$, and, for every $e\in \Lambda$, $x +e \notin \N(T)$. Trivially, $x=\m (S)$. Since $T$ is an $\CaA$-semigroup and $\N(S)=\N(T)\cup\{x\}$, we conclude that $S\in \CaA(f)$.
\end{proof}

From the preceding results, we propose Algorithm \ref{ComputeCaA(Fb=f)} to construct the set $\CaA(f)$.

\begin{algorithm}[H]
\caption{Computing the set $\CaA(f)$.}\label{ComputeCaA(Fb=f)}
\KwIn{A non-negative integer cone $\CaC$ and $f\in\CaC\setminus\{0\}$.}
\KwOut{The set  $\CaA(f)$.}
$A \leftarrow \{\Delta(f)\}$\;
$X\leftarrow A$\;

\While {$A\ne\emptyset$}{
    $Y \leftarrow \emptyset$\;
    $B\leftarrow A$\;
    \While{$B\neq \emptyset$ }
    {$T \leftarrow \text{First}(B)$\;
$C \leftarrow \{x\in SG(T)\setminus(\Lambda \cup \{f\})\mid x\prec \m(T) \text{ and } x +e \notin \N(T) \text{ for all } e \in \Lambda\}$\; \label{lineaAlg1}
    $Y \leftarrow Y\cup\{T\cup\{x\}\mid x\in C\}$\;
    $B \leftarrow B\setminus \{T\}$\;
    }
    $A \leftarrow Y$\;
    $X \leftarrow X\cup Y$\;
    }
\Return{$X$}
\end{algorithm}

Determining the set $\SG(S)$ could involve a high computational cost, since it requires to verify: $x+s \in S$, and $2x \in S$, for all $x \in \CaH(S)$ and $s \in \msg(S)$ with $s\prec \Fb(S)$.
Moreover, this procedure is repeated for several $\CaC$-semigroups (see Line \ref{lineaAlg1} of Algorithm \ref{ComputeCaA(Fb=f)}). We propose an approach to compute $\SG(S \cup \{x\})$ from $\SG(S)$, for any $x \in \SG(S)$. For this purpose, we introduce the notion of the Apéry set from a classical tool presented in \cite{apery1946branches}. Given a $\CaC$-semigroup $S$ and an element $b\in S\setminus\{0\}$, the Apéry set of $S$ with respect to $b$ is defined as
\[
\Ap(S,b)=\{a\in S\mid a-b\in \CaH(S)\}.
\]
Since $\Ap(S,b)\setminus \{0\} = S\cap \big(\CaH(S)+\{b\}\big)$, it follows that $\Ap(S,b)$ is finite for every $b\in S\setminus \{0\}$, and $\Ap\big(\Delta(f),m\big)=(\CaH(\Delta(f))+\{m\})\cup \{0\}$ where $m=\m(\Delta(f))$.

For any $x,y\in L\subseteq\mathbb N^p$, consider the partial order $x \leq_L y$ if $y - x \in L$. The relation between the Apéry set of $S$ and its pseudo-Frobenius set is shown below.

\begin{proposition}\cite[Proposition 16]{garcia2020pseudo}.
Let $S$ be a $\CaC$-semigroup. Then,
\[
\PF(S)=\{a-b\mid a\in \operatorname{maximals}_{\leq_S}\Ap(S,b) \}.
\]
\end{proposition}

In particular, observe that $\operatorname{maximals}_{\leq_{\Delta(f)}}\left(\Ap\big(\Delta(f),m\big)\right)=\Ap\big(\Delta(f),m\big)$, and recall that $\SG(S)=\{x\in\PF(S)\mid 2x\in S\}$. So, we focus on $\Ap(S,b)$
to determine $\SG(S)$. The following result, inspired by \cite[Lemma 3.5]{moreno2024covariety}, allows us to identify $\SG(S \cup \{x\})$ from $\SG(S)$ using Apéry sets.

\begin{proposition}\label{prop:Ap(Scupx)}
Let $S$ be a $\CaC$-semigroup, $x\in \SG(S)$ and $b\in S\setminus\{0\}$. Then,
\begin{equation}\label{inclusionAp}
\Ap(S\cup\{x\},b)\subseteq \{x\}\sqcup \big(\Ap(S,b)\setminus\{x+b\}\big).
\end{equation}
\end{proposition}

Note that the inclusion \eqref{inclusionAp} is an equality when $x-b\in \CaH(S)$.

\begin{example}\label{Ex:Sec1}
Let $\CaC\subset \mathbb N^2$ be the positive integer cone spanned by $\{(12, 1), (7, 4)\}$. Consider $f=(7,2)\in \CaC$, and the degree lexicographical order. By applying Algorithm \ref{ComputeCaA(Fb=f)}, we obtain the tree shown in Figure \ref{fig:tree}. Note that the root of the tree $G(\CaA(f))$ is the $\CaC$-semigroup $\Delta (f)$, and each vertex is labelled with the corresponding joined special gap. For example, the rightmost node $(3,1)$ in the last level of the tree is the $\CaC$-semigroup $\Delta(f) \cup \{(6,2),(5,1), (4,2), (3,1)\}$.

\begin{figure}[h]
    \centering
    \includegraphics[scale=.4]{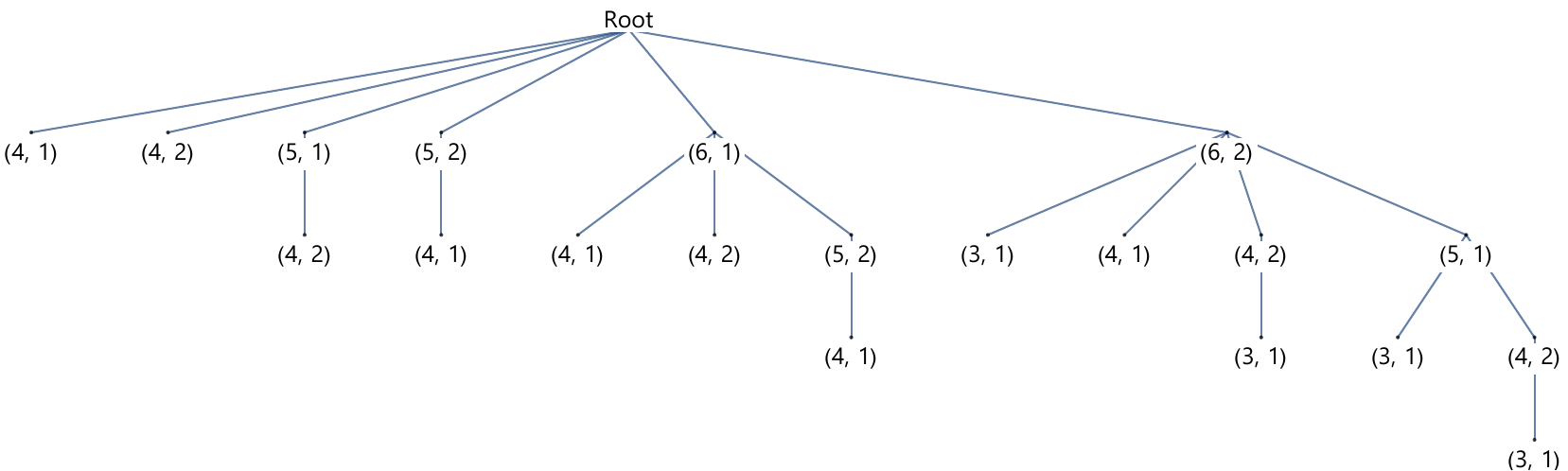}
    \caption{Tree $G(\CaA(f))$ with $f=(7,2)$.}
    \label{fig:tree}
\end{figure}

\end{example}

\section{$\CaA(f)$-systems of generators}\label{Sec2}

Throughout this section, an element $f\in\CaC$ and a monomial order $\preceq$ on $\mathbb N^p$ are fixed. We say that a set $X$ is an $\CaA(f)$-set if $X \cap \Delta(f) = \emptyset$ and there exists $S \in \CaA(f)$ such that $X \subset S$.
Equivalently, $X$ is an $\CaA(f)$-set if and only if there exists $S \in \CaA(f)$ with $X \subseteq \N(S) \setminus \{0\}$.
Furthermore, given an $\CaA(f)$-set $X$, we define $\CaA(f)[X]$ as the intersection of all elements of $\CaA(f)$ containing  $X$.

Since $\CaA(f)$ is a finite covariety, we deduce the following property.

\begin{proposition}\label{prop:CaA(f)[X]}
Let $X$ be an $\CaA(f)$-set. Then, $\CaA(f)[X]$ is the smallest element of $\CaA(f)$ containing $X$, with respect to inclusion.
\end{proposition}

Given an $\CaA(f)$-set $X$, we say that $X$ is an $\CaA(f)$-generator system of $S$ if $X$ is an $\CaA(f)$-set and $S = \CaA(f)[X]$. Moreover, if there is no proper subset $Y \subsetneq X$ such that $\CaA(f)[Y] = S$, we say that $X$ is a minimal $\CaA(f)$-system of generators of $S$.
If there is a unique minimal $\CaA(f)$-system of generators of $S$, we denote it by $\msg_{\CaA(f)}(S)$, and the cardinality of it, called the $\CaA(f)$-rank of $S$, is represented by $\rank_{\CaA (f)}(S)$.

The main goal of this section is to prove that every $S \in \CaA(f)$ admits a unique minimal $\CaA(f)$-system of generators, which we explicitly determine. Subsequently, we characterize those $\CaA$-semigroups of $\CaA(f)$-rank equal to zero, one, and two.

We state the following result based on \cite[Proposition 4.2]{moreno2024covariety}.

\begin{theorem}\label{Thrm:msgCaA(f)}
Let $S\in \CaA(f)$. Then,
\[
\msg_{\CaA (f)}(S)=\N(S)\cap\msg(S).
\]
\end{theorem}

\begin{proof}
Let $S\in \CaA(f)$ and consider $X=\N(S)\cap\msg(S)$. By definition, $X$ is trivially an $\CaA(f)$-set. Let us see that $X$ is a $\CaA(f)$-system of generators of $S$. Since $X\subset S$, then $\CaA(f)[X]\subseteq S$.
For the other inclusion, observe that $X\subset\CaA(f)[X]$ and $\CaA(f)[X]\in\CaA(f)$, which implies that $X\sqcup\Delta(f)\subseteq \CaA(f)[X]$. Consequently, $\msg(S)\subset\CaA(f)[X]$, and thus $S\subseteq \CaA(f)[X]$.  To complete the proof, we show that if $Y$ is an $\CaA(f)$-set such that $\CaA(f)[Y] = S$, then necessarily  $X \subseteq Y$. Suppose that $X \not \subseteq Y$, then $x\in X \setminus Y$ exists. Since $x\in \msg(S)$ and $S\in \CaA(f)$, we have $S\setminus\{x\}\in \CaA(f)$ and $Y\subset S\setminus\{x\}$. Applying Proposition \ref{prop:CaA(f)[X]}, $S=\CaA(f)[Y]\subseteq S\setminus\{x\}$, which is false. Hence, $X$ is the unique minimal $\CaA(f)$-system of generators of $S$.
\end{proof}

As a consequence of Theorem \ref{Thrm:msgCaA(f)}, we obtain the following corollary, which in particular caracterizes the elements $S\in\CaA(f)$ with $\rank_{\CaA(f)}(S)=0$, and gives the key fact for those with $\rank_{\CaA(f)}(S)=1$.

\begin{corollary}\label{coro:rank}
Let $S\in\CaA(f)$. Then, the following conditions hold:
\begin{itemize}
\item $\rank_{\CaA(f)}(S)\leq\e(S)$.
\item $\rank_{\CaA(f)}(S) = 0$ if and only if $S = \Delta(f)$.
\item If $S \ne \Delta(f)$, then $\m(S) \in \msg_{\CaA(f)}(S)$.
\item $\rank_{\CaA(f)}(S) =1$ if and only if $\msg_{\CaA(f)}(S)=\{\m(S)\}$.
\end{itemize}
\end{corollary}

Next, we provide a characterization of the elements $S \in \CaA(f)$ with $\rank_{\CaA(f)}(S) = 1$. Recall that $\Lambda$ is the canonical basis of $\mathbb N^p$. Note that if $\m(S)\in \Lambda$, then $\m(S)$ and $2\m(S)$ are consecutive. So, when $\m(S), 2\m(S)\in \N(S)$, we deduce that $S$ is not an $\CaA$-semigroup. Therefore, any element $S$ in $\CaA(f)$ with $\rank_{\CaA(f)}(S) =1$ and $\m(S)\in \Lambda$ is $\langle m \rangle \cup \Delta(f)$, for some $m\in \Lambda$ and $m\prec f \prec 2m$.

\begin{theorem}\label{thr:caracrank=1}

Let $S$ be a $\CaC$-semigroup and $m, f \in \CaC$ with $m\notin \Lambda$. Then, $S\in \CaA(f)$ with $\rank_{\CaA(f)}(S) =1$ if and only if
\begin{equation}\label{Sdescription}
S = \langle m \rangle \cup \Delta(f),
\end{equation}
where $m \prec f$ and $f\notin \langle m\rangle$.
\end{theorem}

\begin{proof}
Suppose that $S  \in \CaA(f)$ and $\rank_{\CaA(f)}(S) =1$. By Corollary \ref{coro:rank}, we have $S=\CaA(f)[\{\m(S)\}]$ and $S\ne\Delta(f)$. Therefore, $\m(S) \prec f$ and $f\notin\langle \m(S)\rangle$, which implies that $S = \langle \m(S) \rangle \cup \Delta(f)$. Conversely, consider the $\CaC$-semigroup $S$ described in \eqref{Sdescription}. Clearly, $S \in \CaA(f)$, $S \neq \Delta(f)$, and any element of $ \CaA(f)$ that contains $\{m\}$ also contains $S$. By Proposition \ref{prop:CaA(f)[X]}, we deduce that $S=\CaA(f)[\{m\}]$, and applying Corollary \ref{coro:rank}, we obtain $\rank_{\CaA(f)}(S) =1$.
\end{proof}

The $\CaA$-semigroup presented in \eqref{Sdescription} can be described in greater detail. For this purpose, we introduce some notation related to monomial orders on $\mathbb{N}^p$. Any such order $\preceq$ can be defined using $p$ distinct maps $\pi_i: \mathbb N^p \longrightarrow\mathbb N$, where each $\pi_i$ is defined via $\pi_i(x_1,\ldots ,x_p)=\sum_{j\in[p]}a_{ij}x_j$, such that $a_{ij}$ are positive real numbers for every $i,j\in[p]$, and the matrix $M=(a_{ij})$ is non-singular (for further details, see \cite{cox1997ideals, robbiano1986theory}). Some results in this work require that the set $\{x\in \CaC\mid x\preceq k\}$ is finite for any $k\in \CaC$, especially when $\CaC=\mathbb N^p$. Hence, we assume that the fixed monomial order satisfies such a property, and that the coefficients of $\pi_1(x_1,\ldots ,x_p)$ are all strictly positive (integer) numbers, that is, the entries of the first row of $M$ are all strictly positive. For example, a graded monomial order can be used.

From now on, $\lfloor \cdot \rfloor$ denotes the floor function, which rounds down to the nearest integer. Therefore, using the above terminology, the $\CaA$-semigroup given in \eqref{Sdescription} corresponds with the set
\[
S =\bigg\{\m(S), 2\m(S), \ldots, \left\lfloor \tfrac{\pi_1(f)}{\pi_1\left(\m(S)\right)} \right\rfloor\m(S) \bigg\}\cup\Delta(f),
\]
which leads to the following result for numerical semigroups. We denote by $\g(S)$ the genus of $S$, that is, the cardinality of $\CaH(S)=\CaC\setminus S$.

\begin{corollary}
If $S\in \CaA(f)$ is a numerical semigroup with $\rank_{\CaA(f)}(S)=1$, then $\g(S) = f - \left\lfloor \frac{f}{\m(S)} \right\rfloor$.
\end{corollary}

\begin{example}
Let  $S=\CaA(14)[\{5\}]=\{0,5, 10, 15, \longrightarrow \}$,  where the symbol $\longrightarrow$ indicates that every integer greater than $15$ belongs to the set. Hence, $\g(S)=14- \left\lfloor \tfrac{14}{5} \right\rfloor=12$.
\end{example}

Consider now the case where $S \in \CaA(f)$ and $\rank_{\CaA(f)}(S) = 2$. We structure our discussion analogously to the case $\rank_{\CaA(f)}(S) = 1$, and provide an explicit description of the $\CaA$-semigroups with $\CaA(f)$-rank equal to two. Additionally, for those that are also numerical semigroups, we present a pseudo-formula for their genus.

Recall that for any $\CaC$-semigroup $S$, the ratio of $S$ is given by $\r(S)=\min_\preceq\big(\msg(S)\setminus\{\m(S)\}\big)$. The next result is a direct consequence of Theorem \ref{Thrm:msgCaA(f)}.

\begin{corollary}\label{coro:rank2}
Let $S\in\CaA(f)$. Then, the following conditions are equivalent:
\begin{itemize}
\item  $\rank_{\CaA(f)}(S)=2$.
\item $\{\m(S), \r(S)\}$ is an $\CaA(f)$-set, and $S=\CaA(f)\big[\{\m(S), \r(S)\}\big]$.
\item $\r(S)\prec f$ and $S=\langle\m(S),\r(S)\rangle \cup \Delta(f)$, where if there exist two consecutive elements of $\langle\m(S),\r(S)\rangle$, then at least one of them is greater than $f$.
\end{itemize}
\end{corollary}

To characterize the elements $S \in \CaA(f)$ with $\rank_{\CaA(f)}(S) = 2$, we briefly explain the conditions under which $\langle m,r\rangle$ is a $\CaC$-semigroup for any $m,r\in \CaC$. The conditions above are outlined as follows.
\begin{itemize}
\item For $p\geq 3$, $\langle m,r\rangle$ is not a $\CaC$-semigroup.
\item For $p=2$, $S=\langle m,r\rangle$ is a $\CaC$-semigroup if and only if $S=\CaC$. Since a $\CaC$-semigroup $S=\langle m,r\rangle \subseteq\CaC\subseteq\mathbb N^2$ has embedding dimension two if and only if the set $\mathbb N^2 \cap \{\lambda_1 a_1+\lambda_2 a_2\mid 0\leq \lambda_i\leq 1,\mbox{ and } a_i=\min_{\leq_{\mathbb N^2}} (\tau_i\cap \CaC) \}=\{0,m,r,m+r\}$ (see \cite{diaz2022characterizing}).
\item For $p=1$, $\langle m,r\rangle$ is a numerical semigroup if and only if $m$ and $r$ are coprime (see \cite[Lemma 2.1]{libroRosales}).
\end{itemize}

We state the following theorem, which provides a characterization of the elements $S \in \CaA(f)$ with $\rank_{\CaA(f)}(S) = 2$ in the case where $\langle m, r \rangle$ is not a $\CaC$-semigroup. From now on, we assume that for any $m, r, f \in \CaC$, if two consecutive elements of $\langle m,r\rangle$ exist, then one of them is greater than $f$.

\begin{theorem}\label{thr:rank2}
Let $S$ be a $\CaC$-semigroup and let $m, r, f \in \CaC$ such that $\langle m, r \rangle$ is not a $\CaC$-semigroup, $m \prec r \prec f$, $r \notin \langle m \rangle$, and $f \notin \langle m, r \rangle$. Then,
\[
S = \langle m, r \rangle \cup \Delta(f) \in \CaA(f) \text{ and } \rank_{\CaA(f)}(S) = 2.
\]
Moreover, every $\CaA$-semigroup with $\CaA(f)$-rank equal to two, such that its multiplicity and ratio do not generate a $\CaC$-semigroup, is of this form.
\end{theorem}

\begin{proof}
Clearly, $S\in \CaA(f)$. By Theorem \ref{Thrm:msgCaA(f)}, it follows that $\msg_{\CaA(f)}(S)=\{m,r\}$. So, $\rank_{\CaA(f)}(S)=2$. Suppose now  that  $T\in \CaA(f)$ with $\rank_{\CaA(f)}(T) =2$ such that $\langle\m(T), \r(T)\rangle$ is not a $\CaC$-semigroup. By Corollary \ref{coro:rank2}, we have $T=\langle \m(T), \r(T) \rangle \cup \Delta(f)$. Trivially,  $\m(T)\prec \r(T) \prec f$, $ \r(T)\notin \langle \m(T) \rangle$ and, $f \notin \langle \m(T),  \r(T) \rangle$.
\end{proof}

Theorem \ref{thr:rank2} does not analyze all numerical semigroups with $\CaA(f)$-rank equal to two. Specifically, the case where $m$ and $r$ are coprime (i.e., when $\langle m, r \rangle$ is a numerical semigroup) has not been considered. Notice that any numerical semigroup $S$ with multiplicity equal to two is minimally generated by $\{2,\Fb(S)+2\}$. Thus, $S$ is an $\CaA$-semigroup with $\rank_{\CaA(\Fb(S))}(S)=1$. Moreover, the multiplicity of any numerical $\CaA$-semigroup with $\CaA(f)$-rank equal to two is greater than or equal to three. In general, affine semigroups require more consideration due to the greater number of cases that distinguish them from numerical semigroups. However, as we have seen, the situation is reversed. This is why we rely on the Apéry set to address all the possibilities for numerical semigroups.

The following two propositions are formulated specifically for them. We recall a well-known characterization of the Apéry set of a numerical semigroup.

\begin{proposition}\cite[Lemma 2.4]{libroRosales}.\label{lem:Ap_congruencias}
Let $S$ be a numerical semigroup and $b\in S\setminus\{0\}$. Then,
\[
\Ap(S,b)=\{0=w(0), w(1), \ldots, w(b-1)\},
\]
where $w(i)$ is the least element of $S$ such that $w(i)\equiv i \mod b$, for all $i\in\{0\}\cup [b-1]$.
\end{proposition}

\begin{proposition}\label{prop:carcApN}
Let $S$ be a numerical semigroup, $b\in S\setminus\{0\}$, $\Ap(S,b)=\{0=w(0), w(1), \ldots, w(b-1)\}$ and $X=\{i\in [b-1]\mid w(i)< \Fb(S)\}$. Then, $S\in\CaA(f)$ if and only if $X\cup\{0,b\}$ does not contain consecutive natural numbers.
\end{proposition}

\begin{proof}
Suppose that $S$ is a numerical semigroup such that $S\notin \CaA(f)$. Therefore, there exists $s, s+1\in \N(S) \cap \Ap(S, f+1)$. By Proposition \ref{lem:Ap_congruencias}, we have $s=w(i)$ and $s+1=w(j)$ for some $i,j\in [f]\cup\{0\}$, which implies that $i,f\in X\cup\{0,f\}$. Since $w(j)= w(i)+1\equiv w(i+1)\mod (f+1)$, we conclude that $j=i+1$.
Conversely, assume that there exist $i,i+1\in X\cup\{0,b\}$, and that $w(b)=w(0)$. We have $w(i)+1 \equiv w(i+1)$ and $w(i+1)-1\equiv w(i) \mod b$. If $w(i+1)< w(i)$, then $w(i)+1\in S$. Otherwise, $w(i+1)-1\in S$. Hence, $S\notin \CaA(f)$.
\end{proof}

The following theorem provides an analogous result to Theorem \ref{thr:rank2} for those numerical semigroups whose multiplicity and ratio are coprime.

\begin{theorem}\label{thr:rank2numerical}
Let $m, r, f \in \mathbb N$ with $3\le m<  r < f$, and such that $\langle m, r \rangle$ is a numerical semigroup, $f \notin \langle m, r \rangle$, and $X=\left\{\lambda r \mod m \mid \lambda\in \big[ \lfloor\tfrac{f-1}{r}\rfloor \big]\right\}\cup\{0,m\}$. If $X$ does not contain consecutive natural numbers, then
\[
S = \langle m, r \rangle \cup \{f+1,f+2, \longrightarrow\} \in \CaA(f) \text{ with } \rank_{\CaA(f)}(S) = 2.
\]
Moreover, every numerical $\CaA$-semigroup with $\CaA(f)$-rank equal to two, such that its multiplicity and ratio generate a numerical semigroup, is of this form.
\end{theorem}

\begin{proof}
Let $S= \langle m, r \rangle \cup \{f+1,f+2, \longrightarrow\}$. Since $\langle m, r \rangle$ is a numerical semigroup, it follows that $\m(S)=m$ and $\r(S)=r$. By \cite[Example 2.22]{libroRosales} we obtain that
\begin{align*}
\{ x \in \Ap(S, m) \mid x < f \} &= \{ x \in \Ap\big( \langle m, r \rangle, m \big) \mid x < f \} \\
&= \left\{ \lambda r \mid \lambda \in \{0\} \cup [ \lfloor\tfrac{f-1}{r}\rfloor ] \right\}.
\end{align*}
Therefore, $X\cup \{0, m\}=\{i\in [m-1]\mid w(i)< f\}\cup\{0,m\}$. By hypothesis, $X\cup\{0,m\}$ does not contain consecutive natural numbers. Applying Proposition \ref{prop:carcApN}, we deduce that $S\in \CaA(f)$. Finally, by Theorem \ref{Thrm:msgCaA(f)}, we obtain that  $\msg_{\CaA(f)}(S)=\{m,r\}$, and thus $\rank_{\CaA(f)}(S)=2$. Suppose now that $T\in \CaA(f)$ with $\rank_{\CaA(f)}(T) =2$, and that $\langle\m(T), \r(T)\rangle$ is a numerical semigroup. By Corollary \ref{coro:rank2}, we have $T=\langle \m(T), \r(T) \rangle \cup \{f+1, \longrightarrow\}$. Trivially,  $\m(T)< \r(T) < f$, $ \r(T)\notin \langle \m(T) \rangle$ and, $f \notin \langle \m(T),  \r(T) \rangle$.
\end{proof}

Below, we present two examples to illustrate Theorem \ref{thr:rank2numerical} and emphasize the importance of the structure of the set $X$.

\begin{example}\label{ex:NScoprime}
Let $m=10$, $r=13$ and $f=27$. Observe that $m,r$ and $f$ satisfy the conditions of Theorem \ref{thr:rank2numerical}. In this case  $\lfloor\tfrac{26}{13}\rfloor= 2$ and $X=\{13 \mod10, 26\mod 10\}\cup \{0, 10\}=\{0,3,6,10\}$, which does not contain consecutive natural numbers. Since the conditions given in Theorem \ref{thr:rank2numerical} hold, we deduce that $S=\langle 10, 13 \rangle \cup \{ 28, \rightarrow \} \in \CaA(27)$ and $\rank_{\CaA(27)}(S)=2$.
\end{example}

\begin{example}\label{ex:NSnon-coprime}
Let $m=7$, $r=11$ and $f=27$. Again, $m,r$ and $f$ verify the conditions of Theorem \ref{thr:rank2numerical}. We obtain that $\lfloor\tfrac{26}{11}\rfloor= 2$ and $X=\{11 \mod7, 22\mod 7\}\cup \{0, 7\}=\{0, 1, 2,7\}$, which  contains consecutive natural numbers. Applying Theorem \ref{thr:rank2numerical}, we conclude that there exist no numerical semigroup $S\in \CaA(27)$ with $\rank_{\CaA(27)}(S)=2$ satisfying $\m(S)=7$ and $\r(S)=11$.
\end{example}

At the end of the section, we provide a pseudo-formula for computing the genus of any numerical semigroup with $A(f)$-rank equal to two. It is a classical result that the Apéry set of a numerical semigroup, with respect to a non-zero element, allows us to compute its genus, as stated below.

\begin{proposition}\cite[Proposition 2.12]{libroRosales}.\label{prop:formulaGENUS}
Let $S$ be a numerical semigroup and $b\in S\setminus\{0\}$. Then, $\g(S)=\frac{1}{b}(\sum_{w\in \Ap(S,b)} w)-\frac{b-1}{2}$.
\end{proposition}

From the previous result, our objective is reduced to study the Apéry set. The next proposition describes explicitly the set $\Ap(S,\m(S))$. Let $B$ be the set $\big\{\lambda \r(S) \mid \lambda\in \{0\} \cup [l] \big\}$, where
$l=\big\lfloor\frac{\Fb(S)-1}{\r(S)}\big\rfloor$ if $\m(S)$ and $\r(S)$ are coprime, or $l = \max \left\{ k \in \big[\tfrac{\m(S)}{\gcd\left(\m(S), \r(S)\right)}-1\big] \mid k \r(S) < \Fb(S) \right\}$, otherwise.

\begin{proposition}\label{prop:formulaAp}
Let $S\in \CaA(f)$ be a numerical semigroup with $\rank_{\CaA(f)}(S)=2$. Then, $\Ap(S,\m(S))$ is the disjoint union of $B$ and the set
$$
\bigcup_{i\in [\m(S)]}\big\{\Fb(S) + i \mid  \Fb(S) + i \not\equiv b \mod \m(S),\,\forall b\in B\big\}.
$$
\end{proposition}

\begin{proof}
Let $S \in \CaA(f)$ be a numerical semigroup with $\rank_{\CaA(f)}(S) = 2$. By Theorems~\ref{thr:rank2} and \ref{thr:rank2numerical},
\[
S = \langle \m(S), \r(S) \rangle \cup \{\Fb(S),\longrightarrow \}.
\]
For the sake of simplicity, we write $m, r$ and $f$ instead of $\m(S), \r(S)$ and $\Fb(S)$, respectively.
By Proposition \ref{lem:Ap_congruencias}, we know that $\Ap(S,m)=\{0=w(0), w(1), \ldots, w(m-1)\}$, where $w(i)$ is the least element of $S$ such that $w(i)\equiv i \mod m$, for all $i\in\{0\}\cup [m-1]$. Suppose that $ \langle \m(S), \r(S) \rangle$ is not a numerical semigroup and let $d=\gcd(m,r)$. Notice that for every natural number $k$,
\[
(k\tfrac{m}{d}+i)r\equiv k\tfrac{m}{d}d\tfrac{r}{d}+ ir\equiv km\tfrac{r}{d}+ir\equiv ir\mod m,
\]
for all $i\in \{0\}\cup[\tfrac{m}{d}-1]$. Thus, defining $l = \max \left\{ k \in \left[\tfrac{m}{d} - 1\right] \mid k r < f \right\}$, we obtain that $B=\big\{\lambda \r(S) \mid \lambda\in \{0\} \cup [l]\big\}\subset \Ap(S,m)$.
Now, let $x\in \Ap(S,m)\setminus B$. So, $x>\Fb(S)$, which implies that $x=\Fb(S)+i$ for some $i\in [m]$. We conclude that
\[
\Ap(S,m)=B \sqcup\big\{\Fb(S)+i\mid\Fb(S)+i\not\equiv b, \text{ for any }b\in B\big\}.
\]
If $\langle m, r \rangle$ is a numerical semigroup, then by Proposition~\ref{prop:carcApN} and Theorem~\ref{thr:rank2numerical}, it follows that the Apéry set admits the same decomposition with $l = \left\lfloor \tfrac{f-1}{r} \right\rfloor$.
\end{proof}

Two examples showing the application of Propositions \ref{prop:formulaGENUS} and \ref{prop:formulaAp}.

\begin{example}
Let $S=\CaA(27)[\{10,13\}]$ be the numerical semigroups given in Example \ref{ex:NScoprime}. From Proposition \ref{prop:formulaAp}, we obtain that \[
\Ap(S, 10)=\{0,13,26,28,29,31,32,34,35,37\}.\]
By applying Proposition \ref{prop:formulaGENUS}, we conclude that \[
\g(S)=\tfrac{1}{10}(0+13+26+28+29+31+32+34+35+37)-\frac{9}{2}=\frac{265}{10}-\frac{9}{2}=22.
\]
\end{example}

\begin{example}
Let $S=\langle 10,14\rangle\{84,\longrightarrow \}=\CaA(83)[\{10,14\}]$ be a numerical semigroup. Thus, $\Ap(S, 10)=\{0,14,28,42, 56, 85, 87, 89, 91, 93\}$, and then $\g(S)=54$.
\end{example}

\section{A partition of $\CaA(f)$}\label{Sec3}

Let $m, f\in \CaC$. We denote by $\CaA(f,m)$ the set of all possibles $\CaA$-semigroups $S$ such that $\Fb(S)=f$ and $\m(S)=m$.
In this section, we provide an alternative procedure to compute $\CaA(f)$, using the sets $\CaA(f,m)$. Moreover, we show that $\CaA(f,m)$ can be arranged in a rooted tree. Again, fix a monomial order $\preceq$ on $\mathbb N^p$.

Given $f$ and $m$ in a cone $\CaC$ with $f\notin\langle m \rangle$, we define
$$\Delta(f,m)=\langle m\rangle\cup \Delta(f).$$
The following result determines under what conditions $\CaA(f,m)\ne\emptyset$.

\begin{proposition}
Let $m, f\in \CaC$.  The set  $\CaA(f,m)\ne\emptyset$ if and only if
\begin{itemize}
\item $m\prec f$, or there does not exist any $z\in \CaC$ such that $f\prec z \prec m$.
\item  $f\notin\langle m\rangle$.
\end{itemize}
\end{proposition}

\begin{proof}
Assume $S\in \CaA(f,m)$. Since $f$ is the Frobenius element of $S$, $f\notin \langle m\rangle$ follows. If $m\prec f$, we are done. Otherwise, suppose there exists $z\in \CaC$ such that $f\prec z \prec m$. By the definition of the Frobenius element, it follows that $z\in S$, which contradicts the definition of multiplicity. Conversely, suppose that $f\notin\langle m\rangle$. If $m\prec f$, then $\Delta(f,m)\in \CaA(f,m)$ by Theorem \ref{thr:caracrank=1}. If there does not exist $z\in \CaC$ such that $f\prec z \prec m$, then $\Delta(f)\in \CaA(f,m)$.
\end{proof}

Fix  $f\in \CaC$. We define,
\begin{equation}\label{M_f}
\M(f)=\{m\in \CaC\setminus\{0\}\mid m\prec f, \text{ or }\nexists\, z\in \CaC \text{ such that } f\prec z \prec m\}.
\end{equation}
Observe that $\CaA(f)=\cup_{m\in \M(f)}\CaA(f,m)$, and thus the family $\{ \CaA(f,m) \mid m\in \M(f)\}$ forms a partition of $\CaA(f)$.

\begin{example}\label{ex:M(f)}
Consider the monomial order, the element $f\in\mathbb N^2$, and the cone $\CaC$ given by Example \ref{Ex:Sec1}. Figure \ref{fig:Delta(f)} illustrates the set $\Delta(f)$. The empty circles represent the gaps of $\Delta(f)$, the blue squares denote its minimal generators, and the red circles represent some elements in it.
\begin{figure}[h]
    \centering
    \includegraphics[scale=.3]{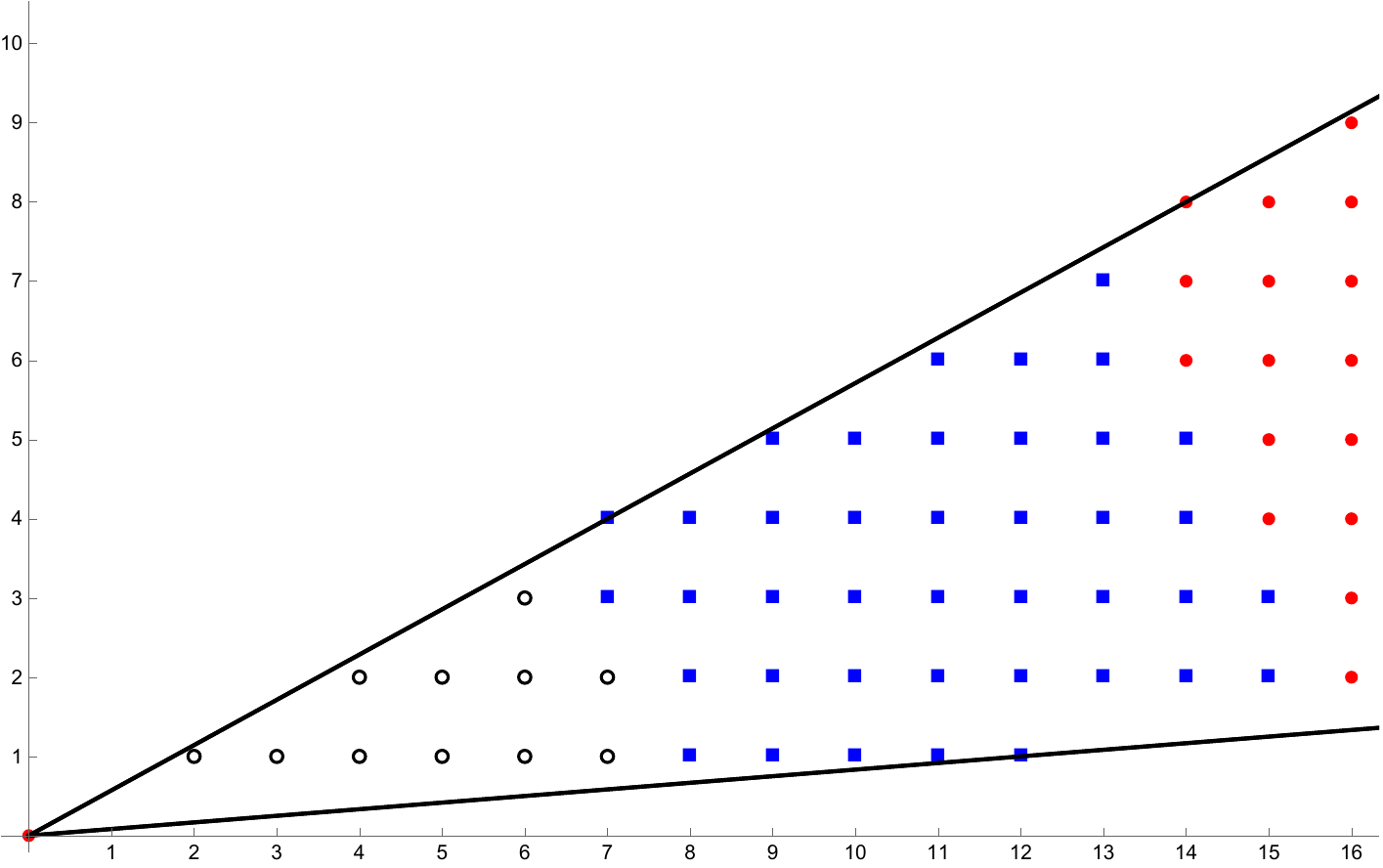}
    \caption{A graphical example of a $\Delta(f)$.}
    \label{fig:Delta(f)}
\end{figure}
So, the set $\M(f)$ is $\{(2,1), (3,1),\ldots ,(8,1),(3,2),\ldots ,(6,2),(6,3)\}$.

\end{example}

We adopt the terminology established in \cite{moreno2024ratio}, extending it from numerical semigroups to $\CaC$-semigroups. A non-empty family $\mathcal{R}$ of $\CaC$-semigroups is a ratio-covariety if verifies the following conditions:
\begin{itemize}
\item[(i)] There exists a minimal element in $\mathcal{R}$ with respect to inclusion, denoted by $\min_{\subseteq}\mathcal{R}$.
\item[(ii)] If $S,T\in \mathcal{R}$, then $S\cap T\in \mathcal{R}$.
\item[(iii)] If $S\in \mathcal{R}\setminus\{\min_{\subseteq}\mathcal R\}$, then $S\setminus\{\r(S)\}\in \mathcal{R}$.
\end{itemize}

We obtain the following result, inspired by Proposition \ref{A(f)covariedad}.

\begin{proposition}\label{prop:A(f,m)ratio-covariety}
Let $f\in \CaC$ and $m\in \M(f)$. Then, the set $\CaA(f,m)$ is a ratio-covariety, and $\Delta(f,m)$ is the minimum with respect to the inclusion.
\end{proposition}

Analogously to Section \ref{Sec1}, we define the graph $G(\CaA(f,m))$ with vertex set $\CaA(f,m)$, where an edge connects distinct elements $S,T \in \CaA(f,m)$ if and only if $T = S \setminus \{\r(S)\}$. As a consequence of Proposition \ref{prop:A(f,m)ratio-covariety} and \cite[Proposition 3]{moreno2024ratio}, we deduce the following result.

\begin{proposition}\label{prop:A(f,m)tree}
The graph $G(\CaA(f,m))$ is a tree with root $\Delta(f,m)$.
\end{proposition}

\begin{proof}
Similarly to the proof of Proposition \ref{prop:grafoA_f}, replacing the multiplicity with the ratio.
\end{proof}

Based on \cite[Proposition 5]{moreno2024ratio}, and using Proposition~\ref{prop:A(f,m)tree}, we obtain the next result.

\begin{proposition}\label{prop:child_G(A(f,m))}
Let $f\in \CaC$, $m\in \M(f)$ and $T\in\CaA(f,m)$. The set of children of $T$ in the tree $G(\CaA(f,m))$ is
\[
\{T\cup\{x\}\mid x\in Z, \text{ and } \m(T)\prec x\prec \r(T)\},
\]
with $Z=\{x\in \SG(T)\setminus\{f\}\mid x\pm e\notin \N(T), \text{ for all } e\in \Lambda \}$.
\end{proposition}

\begin{proof}
Let $S$ be a child of $T$. So, $S=T\cup\{x\}\in \CaA(f,m)$ with $x=\r(S)$, which implies that $x\in \SG(T)$ and $ \m(T)\prec x\prec \r(T)$. Since $S,T\in \CaA(f,m)$, then $x\pm e\notin \N(T)$ for any  $e\in \Lambda$. Now, let $x\in Z$ such that $\m(T)\prec x\prec \r(T)$. Therefore, $T\cup\{x\}\in \CaA(f,m)$ and $\r(T\cup \{x\})=x$. Hence, $T\cup\{x\}$ is a child of $T$.
\end{proof}

Proposition \ref{prop:child_G(A(f,m))} provides an algorithmic procedure for computing the set $\CaA(f,m)$ (Algorithm \ref{ComputeCaA(f,m)}).

\begin{algorithm}[H]
\caption{Computing the set $\CaA(f,m)$.}\label{ComputeCaA(f,m)}
\KwIn{A non-negative integer cone $\CaC$, $f\in\CaC\setminus\{0\}$ and $m\in \M(f)$.}
\KwOut{The set  $\CaA(f,m)$.}
$A \leftarrow \{\Delta(f,m)\}$\;
$X\leftarrow A$\;

\While {$A\ne\emptyset$}{
    $Y \leftarrow \emptyset$\;
    $B\leftarrow A$\;
    \While{$B\neq \emptyset$ }
    {$T \leftarrow \text{First}(B)$\;
    $Z \leftarrow \{x\in SG(T)\setminus\{f\} \mid  m\prec x\prec \r(T)\text{ and }x\pm e\notin \N(T), \forall e\in \Lambda\}$\;
    $Y \leftarrow Y\cup\{T\cup\{x\}\mid x\in Z\}$\;
    $B \leftarrow B\setminus \{T\}$\;
    }
    $A \leftarrow Y$\;
    $X \leftarrow X\cup Y$\;
    }
\Return{$X$}
\end{algorithm}

To fulfil the goal of this section, we now turn to the computation of $\CaA(f)$ using its partition and applying Algorithm \ref{ComputeCaA(f,m)}, as exemplified below.

\begin{example}
Follow with Example \ref{ex:M(f)}, and fix $m=(3,1)\in \M(f)$. Thus, $\Delta(f,m)=\langle (3,1) \rangle \cup \Delta \big((7,2)\big)$. Then, applying Algorithm \ref{ComputeCaA(f,m)}, we obtain three $\CaA$-semigroups belonging to $\CaA(f,m)$,
$$\CaA(f,m)=\Big\{J \cup\Delta(f,m) \mid J\in \big\{\{(4,2)\},\{(5,1)\},\{(4,2),(5,1)\}   \big\} \Big\} .$$
Using the same procedure on each element of $\M(f)$, we again obtain the set $\CaA(f)$ shown in Figure \ref{fig:tree}.
\end{example}

In \cite[Section 6]{rosales2023numerical}, an algorithm is provided for computing the set $\CaA(f,m)$ for a numerical semigroup according to whether the Frobenius number is greater or less than $2m$. The following section generalizes some results in that paper, and introduces a partition of the set $\CaA(f,m)$.

\section{A partition of $\CaA(f,m)$}\label{Sec4}

Let $\CaC\subseteq \mathbb N^p$ be a non-negative integer cone and $\preceq$ a monomial order on $\mathbb N^p$. Consider that the linear function $\pi(x_1,\ldots,x_p)=\sum_{i\in[p]}a_{i}x_i$ defines the first row of a fixed non-singular matrix associated with the given monomial order (recall that we consider that $a_{i}>0$ for all $i$). Fix in this section two elements $f,m\in \CaC$, and consider $\mathcal I (f,m)=\{x\in \CaC \mid m\preceq x \prec f\}$.

When $f\prec 2m$, the following result generalizes \cite[Proposition 46]{rosales2023numerical}. The set  $M(f)$ is defined in \eqref{M_f}.

\begin{proposition}
Let $f,m\in\CaC\setminus\{0\}$ be such that $m\in M(f)$, and $f\prec 2m$. Then, it is satisfied:

\begin{enumerate}
\item If $f\prec m$, then $\CaA(f,m) =\{\Delta (f)\}$.
\item If $\mathcal I (f,m)=\{m=m_1\prec \cdots \prec m_h\}$ with $h\ge 1$, then
\begin{multline*}
\CaA(f,m) =\big\{ \{m\}\cup A \cup \Delta (f) \mid m+e\notin A,\, \forall e\in \Lambda,\\
\text{ and } A\subseteq \mathcal I (f,m) \text{ contains no consecutive elements} \big\}.
\end{multline*}
\end{enumerate}
\end{proposition}

From now on, in this section, assume that $f\succ 2m$. Consider $\bar e=\min_\preceq\big\{e\in \Lambda \mid \pi(e)=\max\{\pi(e')\mid e'\in \Lambda\}\big\}$, and
$$\Upsilon = \{x\in \CaC\mid x\prec f\text{ and } \pi(x)> \pi(f-m-\bar e) \}.$$
The next proposition generalizes \cite[Proposition 47]{rosales2023numerical} from numerical semigroups to $\CaC$-semigroups.

\begin{proposition}\label{prop:Upsilon}
Let $S$ be a $\CaC$-semigroup, $m=\m(S)$, $f=\Fb(S)$, and assume that $f\succ 2m$. Then, $S$ is an $\CaA$-semigroup if and only if these conditions hold:
\begin{enumerate}
\item[\emph{C1.}]\label{C1} If $x,x+e\in \Upsilon$ with $e\in \Lambda$, then $\{x,x+e\}\not\subset S$.
\item[\emph{C2.}] If there exists $y\in S$ such that $\pi(y)=\pi(f-m-\bar e)$, then $y+e\notin S$ for every $e\in \Lambda$.
\end{enumerate}
\end{proposition}

\begin{proof}
Consider the $\CaA$-semigroup $S$. Hence, if there exists some $x\in S$, with $x+e\in S$ for some $e\in \Lambda$, then $x+e\succ f$. Thus, $x+e\notin \Upsilon$, and the condition {\rm C1} follows. To check the condition {\rm C2}, suppose that there exists $y \in S$ with $\pi(y)=\pi(f-m-\bar e)$ such that $y+e\in S$ for some $e\in \Lambda$. So, $y\prec f$, and $\pi(y+e)=\pi(f-m-\bar e+e)< \pi(f)$, a contradiction. Moreover, for any $y\in S$ with $y\prec f$, $y-e\notin S$ for every $e\in \Lambda$.

Conversely, assume the conditions {\rm C1} and {\rm C2} hold, and let $x,x+e\in \CaC$ with $e\in \Lambda$ such that $x,x+e\prec f$. If $\pi(x)$ and $\pi(x+e)$ are strictly greater than $\pi(f-m-\bar e)$, then $x,x+e\in \Upsilon$, and $\{x,x+e\}\not\subset S$. In case that $\pi(x)=\pi(f-m-\bar e)$, then the condition {\rm C2} implies that $x+e\notin S$. Suppose that $\pi(x+e)\leq \pi(f-m-\bar e)$, and set $k=\max\{n\in \mathbb N\mid x+e+nm\prec f\}$. Trivially, $x+km, x+e+km\prec f$. Since $\pi(x+e+(k+1)m)\geq \pi(f)$, we obtain that $\pi(x+e+km)\geq \pi(f-m) > \pi(f-m-\bar e)$, and then $x+e+km\in \Upsilon$. If we assume that $\pi(x+km)< \pi(f-m-\bar e)$, then $\pi(f)> \pi(x+(k+1)m+\bar e)\ge  \pi(x+(k+1)m+e) \ge \pi(f)$, which is a contradiction. Thus, $\pi(x+km)\geq \pi(f-m-\bar e)$. In the case $\pi(x+km)> \pi(f-m-\bar e)$, we can assert that $x+km\in \Upsilon$, and condition {\rm C1} ensures that $\{x+km,x+e+km\}\not\subset S$. It means that $\{x,x+e\}\not\subset S$. Otherwise, if $\pi(x+km)=\pi(f-m-\bar e)$, then the element $x+e+km\notin S$, so $x+e\notin S$.
Therefore, no consecutive elements smaller than $f$ belong to $S$. That is, $S$ is an $\CaA$-semigroup.
\end{proof}

\begin{remark}
Note that since $f\succ 2m\succ 0$, then $\pi(f-m-\bar e)\geq \sum_{i\in [p]} a_im_i-a_j$ with $a_j= \pi(\bar e)$. However, this lower bound can be non-positive. For example, consider $\pi(x,y)=x+4y$, and $S$ the $\CaC$-semigroup minimally generated by $\{(2,0),(5,0),(2,1),(3,1)\}$. We have that $\m(S)=(2,0)$, $\Fb(S)=(3,0)$ and $\bar e=(0,1)$. Thus, $\pi(f-m-\bar e)=-3$. Hence, in that case, $\Upsilon= \{x\in \CaC\mid x\prec f\}$, and the previous proposition is trivial and not interesting. Therefore, the result is generally more useful when the set $\Upsilon$ is minimized, that is, when $\pi(f-m-\bar e)$ is maximized.
\end{remark}

Given $S\in \CaA(f,m)$, $B(S)$ denotes the elements in $\Upsilon$ that do not belong to $S$. Proposition \ref{prop:Upsilon} implies that if $x,x+e\in \Upsilon$, for some $e\in \Lambda$, then $\{x,x+e\}\cap B(S)$ is not empty.
We denote by $B(f,m)$ the set
\begin{multline*}
\big\{B\subseteq \Upsilon \setminus \langle m\rangle \mid \{x,x+e\}\cap B\neq \emptyset \text{ when }\{x,x+e\}\subseteq \Upsilon, \\ f\notin \langle m\cup (\Upsilon\setminus B)\rangle,
\text{ and } \langle m\cup (\Upsilon\setminus B)\rangle\cap B=\emptyset
\big\}.
\end{multline*}
For any $B\in B(f,m)$, $\CaA(f,m,B)=\{S\in \CaA(f,m)\mid B(S)=B\}$. Note that, if $B\in B(f,m)$, then $\CaA(f,m,B)$ is not empty since
$$\Delta(f,m,B)=\langle m \rangle \cup \big( \Upsilon\setminus B  \big)\cup \Delta (f)\in \CaA(f,m,B).$$

\begin{example}\label{ex:Delta(f,m,B)}
Consider $\CaC\subset\mathbb N^2$ the cone with extremal rays determined by $(12,1)$ and $(7,4)$, and fix $f=(9,2)$, $m=(2,1)$, $\preceq$ the degree lexicographic order on $\mathbb N^2$. So, $\pi(x,y)=x+y$, $\bar e=(0,1)$, $\pi(f-m-\bar e)=7$ and
$$\Upsilon=\{(6,2),(6,3), (7,1),(7,2),(7,3),(7,4),(8,1),(8,2),(8,3),(9,1)\}.$$
A set belonging to $B(f,m)$ is $B=\{(6,2),(7,1),(7,3),(8,2),(9,1)\}$, and $\Delta(f,m,B)$ is an $\CaA$-semigroup. Figure \ref{fig:Delta(f,m,B)} illustrates this example. The empty circles represent the gaps of $\Delta(f,m,B)$, the blue squares denote its minimal generators, the red circles represent elements in it, and the set $\Upsilon$ is delimited by the dashed lines.
\begin{figure}[h]
    \centering
    \includegraphics[scale=.4]{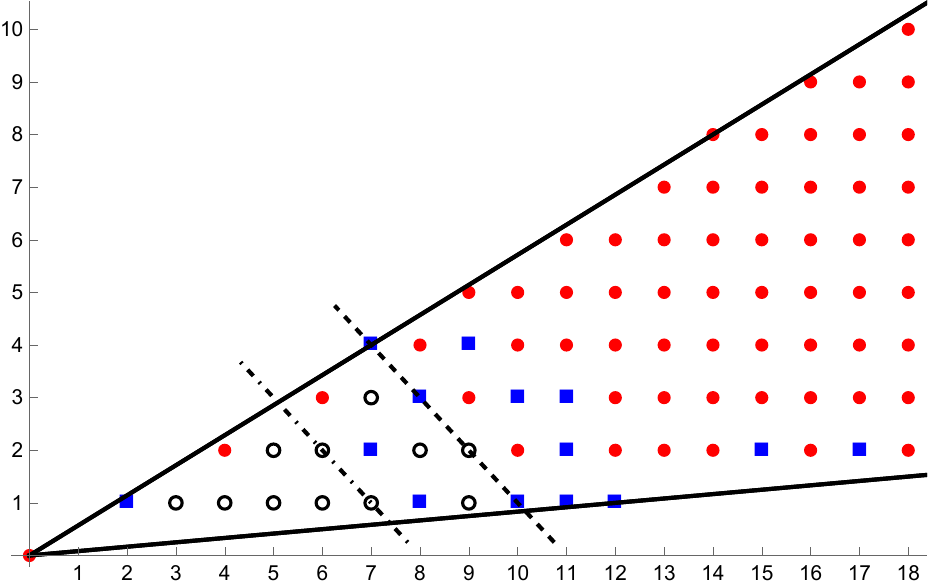}
    \caption{A graphical example of a $\Delta(f,m,B)$.}
    \label{fig:Delta(f,m,B)}
\end{figure}
\end{example}

The result established in \cite[Proposition 48]{rosales2023numerical} can be generalized to $\CaC$-semigroups.

\begin{proposition}
The set $\{\CaA(f,m,B)\mid B\in B(f,m)\}$ is a partition of $\CaA(f,m)$.
\end{proposition}

Now, using a tree, we focus on introducing an algorithm to compute the set $\CaA(f,m,B)$ as in the previous sections. The following lemma gives a root of this tree.

\begin{lemma}
Given $B\in B(f,m)$, $\CaA(f,m,B)$ is a ratio-covariety with minimum element $\Delta(f,m,B)$.
\end{lemma}

We denote by $G(\CaA(f,m,B))$ the graph with vertex set $\CaA(f,m,B)$, and such that $(S,T)\in \CaA(f,m,B)^2$ is an edge if $T=S\setminus \{\r(S)\}$. From the previous lemma, we have that this graph is a tree. This fact allows us to introduce the key to determine the announced algorithm.

\begin{proposition}
The graph $G(\CaA(f,m,B))$ is a tree with root $\CaA(f,m,B)$.
\end{proposition}

\begin{proposition}
Let $B\in B(f,m)$ and $T\in \CaA(f,m,B)$. Then, the set of child of $T$ in $G(\CaA(f,m,B))$ is
$$\big\{T\cup \{x\}\mid x\in \SG(T)\setminus \Upsilon,\, m\prec x \prec \r(T), \text{ and } x\pm e\notin \N(T),\forall e\in\Lambda \big\}.$$
\end{proposition}

\begin{proof}
Analogously to the proof of Proposition \ref{prop:child_G(A(f,m))}.
\end{proof}

\begin{algorithm}[H]
\caption{Computing the set $\CaA(f,m,B)$ with $f\succ 2m$.}\label{ComputeCaA(f,m,B)}
\KwIn{A non-negative integer cone $\CaC$, $f\in\CaC\setminus\{0\}$, $m\in \M(f)$, and $B\in B(f,m)$.}
\KwOut{The set  $\CaA(f,m,B)$.}
$A \leftarrow \{\Delta(f,m,B)\}$\;
$X\leftarrow A$\;

\While {$A\ne\emptyset$}{
    $Y \leftarrow \emptyset$\;
    $C\leftarrow A$\;
    \While{$C\neq \emptyset$ }
    {$T \leftarrow \text{First}(C)$\;
    $Z \leftarrow \{x\in SG(T)\setminus(\Upsilon\cup \{f\}) \mid  \m(T)\prec x\prec \r(T)\text{ and }x\pm e\notin \N(T), \forall e\in \Lambda\}$\;\label{step}
    $Y \leftarrow Y\cup\{T\cup\{x\}\mid x\in Z\}$\;
    $C \leftarrow C\setminus \{T\}$\;
    }
    $A \leftarrow Y$\;
    $X \leftarrow X\cup Y$\;
    }
\Return{$X$}
\end{algorithm}

We give an illustrative example of Algorithm \ref{ComputeCaA(f,m,B)} that is easy to follow.

\begin{example}
Consider $\CaC$, $f$, $m$, $B$ and the fixed monomial order  (degree lexicographical order) given in Example \ref{ex:Delta(f,m,B)}. Note that there is only a $x=(5,1)\in\SG(\Delta(f,m,B))$ satisfying that $m\prec x\prec \r(\Delta(f,m,B))=(7,2)$, and $x\pm e\notin \N(\Delta(f,m,B))$, $\forall e\in \Lambda$. Hence, $\CaA(f,m,B)$ has two elements: $\Delta(f,m,B)$, and the $\CaC$-semigroup showed in Figure \ref{fig:A(f,m,B)}.
\begin{figure}[h]
    \centering
    \includegraphics[scale=.4]{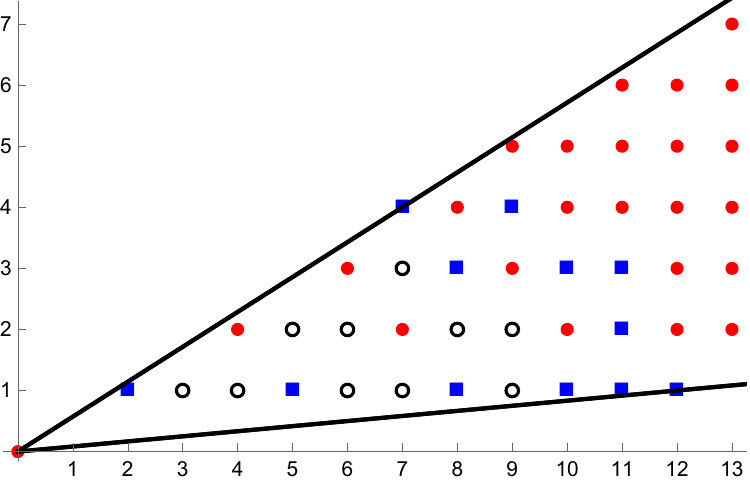}
    \caption{An element in $\CaA(f,m,B)$.}
    \label{fig:A(f,m,B)}
\end{figure}
\end{example}

This section concludes with some considerations on the case of numerical $\CaA$-semigroups.

\begin{remark}
In case that the ambient space is $\mathbb N$, that is, we consider only numerical semigroups, the computation of some sets introduced in this section can be improved. Let $f\in \mathbb N$, and $m\in \M(f)$ with $f>2m$. Hence, the set $\Upsilon$ can be redefined as $\Upsilon=\{f-1,\ldots ,f-m+1\}$, and $B(f,m)$ as $\big\{B\subseteq \Upsilon\setminus \langle m \rangle \mid \{x,x+1\}\cap B\neq \emptyset \text{ when }\{x,x+1\}\subseteq \Upsilon\big\}$. With these new definitions, the previous results are also true, and the step \ref{step} in Algorithm \ref{ComputeCaA(f,m,B)} improves its computational behaviour.
\end{remark}

\section{Numerical $\CaA $-semigroups with maximal embedding dimension}\label{Sec5}

Following the last comment in the previous section, we now focus on numerical semigroups. A numerical semigroup has maximal embedding dimension (or is a $\MED$-semigroup) if $\e(S)=\m(S)$. Fix $f$ and $m$ two different natural numbers such that $f$ is not a multiple of $m$, we denote by $\MED(f,m)$ the set of all numerical semigroups with Frobenius number $f$ and multiplicity $m$.
According to \cite{rosales2023numerical}, we define an $\AMED$-semigroup as a numerical semigroup that is both an $\CaA$-semigroup and a $\MED$-semigroup, and the set $\AMED(f,m)$ corresponds with the set of all  $\AMED$-semigroups with Frobenius number $f$ and multiplicity $m$.

This section aims to provide an algorithm to compute the set $\AMED(f,m)$. In addition, we show two families of $\AMED$-semigroups: Arf and saturated semigroups. We begin by recalling a structural result on the set $\MED(f,m)$,

\begin{proposition}\label{prop:MED}
\cite[Proposition 2.5]{moreno2024ratioMED}
The set $\MED(f, m)$ is a ratio-covariety and $\min_\subseteq\big(\MED(f, m)\big)=\Delta(f,m)$.
\end{proposition}

To study intersections of ratio-covarieties, we will use the following general result.

\begin{proposition}\label{prop:AMEDux}
\cite[Lemma 17]{moreno2024ratio}
Let $\{\mathcal{R}_i\}_{i\in[l]}$ be a family of $l$  ratio-covarieties such that $\min_\subseteq(\mathcal{R}_i)=\Delta$ for all $i\in [l]$. Then,  $\cap_{i\in[l]}\mathcal{R}_i$ is a ratio-covariety and $\Delta$ is its minimum with respect to the inclusion.
\end{proposition}

We obtain the following result as a consequence of Propositions \ref{prop:MED} and \ref{prop:AMEDux}.

\begin{corollary}\label{coro:AMED}
The set $\AMED(f,m)$ forms a ratio-covariety whose minimal element with respect to inclusion is $\Delta(f,m)$.
\end{corollary}

Consider the graph $G\big(\AMED(f,m)\big)$ whose vertex set is $\AMED(f,m)$, and a pair $(S, T) \in \AMED^2 (f,m)$ is  an edge if and only if $T = S \setminus \{r(S)\}$. We now present the results that will allow us to construct the graph $G\big(\AMED(f,m)\big)$.

\begin{theorem}\label{thr:AMED}
The graph $G\big(\AMED(f,m)\big)$ is a tree with root  $\Delta(f,m)$. The set of child of $T$ in $G\big(\AMED(f,m)\big)$ is
\[
\big\{T\cup \{x\} \in \operatorname{MED}(f, m) \mid x\in \SG(T),\, \m(T)<x<\r(T),\; \{x+1,x-1\}\subset\CaH(T)\big\}.
\]
\end{theorem}

\begin{proof}
Immediate from Corollary~\ref{coro:AMED} and \cite[Propositions 3 and 4]{moreno2024ratio}.
\end{proof}

In order to determine whether $T \cup \{x\}$ belongs to $\operatorname{MED}(f, m)$.  We recall that a numerical semigroup $S$ is MED  if and only if $\msg(S)=\big(\Ap(S, \m(S))\setminus\{0\}\big)\cup \{\m(S)\}$ (see \cite[Proposition 3.1]{libroRosales}). Consequently, applying Proposition~\ref{prop:Ap(Scupx)} to numerical semigroups, we obtain
\[
\Ap(S\cup \{x\}, \m(S)\big)=\Big(\Ap(S, \m(S)\big)\setminus\{x+\m(S)\}\Big)\cup\{x\},
\]
which allows us to test whether $T \cup \{x\}$ preserves the maximal embedding dimension property.

\begin{algorithm}[H]
\caption{Computing $\AMED $-semigroup with a given Frobenius number and multiplicity.}\label{ComputeAMED(f,m)}
\KwIn{Two non-zero natural numbers $f$ and $m$.}
\KwOut{The set  $\AMED(f,m)$.}
\If{$f\in \langle m\rangle$ }{
\Return{$\emptyset$}
}
$A \leftarrow \{\Delta(f,m)\}$\;
$X\leftarrow A$\;
\While {$A\ne\emptyset$}{
    $Y \leftarrow \emptyset$\;
    $C\leftarrow A$\;
    \While{$C\neq \emptyset$ }
    {$T \leftarrow \text{First}(C)$\;
    $Z \leftarrow \{ x\in \SG(T)\mid \m(T)<x<\r(T),\; \{x+1,x-1\}\subset\CaH(T),\, T\cup\{x\}\in \operatorname{MED}(f,m)\}$\;
    $Y \leftarrow Y\cup\{T\cup\{x\}\mid x\in Z\}$\;
    $C \leftarrow C\setminus \{T\}$\;
    }
    $A \leftarrow Y$\;
    $X \leftarrow X\cup Y$\;
    }
\Return{$X$}
\end{algorithm}

We present a straightforward example showing how Algorithm \ref{ComputeAMED(f,m)} works.

\begin{example}
To compute the set $\AMED(13,4)$, note that $\Delta(13,4)=\{0,4,8,12,14, \rightarrow\}$. By applying Algorithm \ref{ComputeAMED(f,m)}, we obtain that $\Delta(13,4)$ has only one child $S_1=\Delta(13,4)\cup\{10\}$. For $S_1$, there is a single child $S_2=\Delta(13,4)\cup\{6,10\}$ which has no further descendants.
\end{example}

To conclude this section, we present different families of $\MED$-semigroups. A numerical semigroup $S$ is an Arf semigroup if, for any $x, y, z \in S$ with $z \geq y \geq x$, it follows that $x + y - z \in S$. We denote by $\operatorname{Arf}(f)$ the set of Arf numerical semigroups with Frobenius number equal to $f$.
A numerical semigroup $S$ is saturated if for any $s, s_1,\ldots , s_r\in S$ such that $s_i\leq s$ for all $i\in[r]$, and $z_1,\ldots, z_r\in \mathbb Z$ such that $z_1s_1+\cdots+
z_rs_r \geq 0$, then $s+z_1s_1+\cdots+
z_rs_r \in S$. We denote by $\operatorname{Sat}(f)$ the set of all saturated numerical semigroups with Frobenius number $f$.
Both classes form subfamilies of $\MED$-semigroups, with the inclusion $\operatorname{Sat}(f) \subseteq \operatorname{Arf}(f)$ (see \cite[Proposition 3.12 and Lemma 3.31]{libroRosales}). The next result shows the connection between these families and $\CaA$-semigroups.

\begin{proposition}\label{prop:ARF-CaA}
\cite[Proposition 31]{rosales2023numerical}
Every Arf semigroup is an $\CaA$-semigroup.
\end{proposition}

Combining Proposition~\ref{prop:ARF-CaA}, with \cite[Proposition 2.7]{moreno2023set}, and \cite[Proposition 11]{rosales2024covariety}, we obtain the following.

\begin{proposition}
Let $f\in \mathbb N$. Then,
\[
\operatorname{Sat}(f)\subseteq \operatorname{Arf}(f)\subseteq \CaA(f)\]
is a chain of covarieties such that $\Delta(f)=\min_\subseteq (\operatorname{Sat}(f))=\min_\subseteq (\operatorname{Arf}(f))= \min_\subseteq (\CaA(f))$.
\end{proposition}

\subsection*{Funding}
The last author is partially supported by grant PID2022-138906NB-C21 funded by MICIU/AEI/ 10.13039/501100011033 and by ERDF/EU.

Consejería de Universidad, Investigación e Innovación de la Junta de Andalucía project ProyExcel\_00868 and research group FQM343 also partially supported all the authors.

This publication and research have been partially granted by INDESS (Research University Institute for Sustainable Social Development), Universidad de Cádiz, Spain.

\subsection*{Author information}
J.C. Rosales. Departamento de \'{A}lgebra, Universidad de Granada, E-18071 Granada, (Granada, Spain).
E-mail: jrosales@ugr.es.

\medskip

\noindent
R. Tapia-Ramos. Departamento de Matem\'aticas, Universidad de C\'adiz, E-11406 Jerez de la Frontera (C\'{a}diz, Spain).
E-mail: raquel.tapia@uca.es.

\medskip

\noindent
A. Vigneron-Tenorio. Departamento de Matem\'aticas/INDESS (Instituto Universitario para el Desarrollo Social Sostenible), Universidad de C\'adiz, E-11406 Jerez de la Frontera (C\'{a}diz, Spain).
E-mail: alberto.vigneron@uca.es.

\subsection*{Data Availability}
The authors confirm that the data supporting some findings of this study are available within it.

\subsection*{Conflict of Interest}
The authors declare no conflict of interest.


\begin{thebibliography}{20}

\bibitem{apery1946branches}
R. Apéry,
\newblock Sur les branches superlinéaires des courbes algébriques,
\newblock \emph{C. R. Acad. Sci. Paris}, vol.~222, no.~1198, pp.~2000, 1946.

\bibitem{Barucci}
V. Barucci, D. Dobbs, and M. Fontana,
\newblock Maximality properties in numerical semigroups and applications to one-dimensional analytically irreducible local domains,
\newblock \emph{Memoirs of the American Mathematical Society}, vol.~125, no.~598, pp.~1--78, 1997.

\bibitem{cox1997ideals}
D. Cox, J. Little, D. O'Shea, and M. Sweedler,
\newblock \emph{Ideals, Varieties, and Algorithms}, 3rd ed.,
\newblock Springer, 1997.

\bibitem{diaz2022characterizing}
J.~D. Díaz-Ramírez, J.~I. García-García, D. Marín-Aragón, and A. Vigneron-Tenorio,
\newblock Characterizing affine $\CaC$-semigroups,
\newblock \emph{Ricerche di Matematica}, vol.~71, no.~1, pp.~283--296, 2022.

\bibitem{garcia2020pseudo}
J.~I. García-García, I. Ojeda, J.~C. Rosales, and A. Vigneron-Tenorio,
\newblock On pseudo-Frobenius elements of submonoids of $\mathbb{N}^d$,
\newblock \emph{Collectanea Mathematica}, vol.~71, no.~1, pp.~189--204, 2020.


\bibitem{Lipman}
J. Lipman,
\newblock Stable ideals and Arf rings,
\newblock \emph{American Journal of Mathematics}, vol.~93, no.~3, pp.~649--685, 1971.

\bibitem{moreno2023set}
M.~A. Moreno-Frías, and J.~C. Rosales,
\newblock The set of Arf numerical semigroups with given Frobenius number,
\newblock \emph{Turkish Journal of Mathematics}, vol.~47, no.~5, pp.~1392--1405, 2023.

\bibitem{moreno2024covariety}
M.~A. Moreno-Frías, and J.~C. Rosales,
\newblock The covariety of numerical semigroups with fixed Frobenius number,
\newblock \emph{Journal of Algebraic Combinatorics}, vol.~60, no.~2, pp.~555--568, 2024.

\bibitem{moreno2024ratio}
M.~A. Moreno-Frías, and J.~C. Rosales,
\newblock Ratio-covarieties of numerical semigroups,
\newblock \emph{Axioms}, vol.~13, no.~3, pp.~193, 2024.

\bibitem{moreno2024ratioMED}
M.~A. Moreno-Frías, and J.~C. Rosales,
\newblock The ratio-covariety of numerical semigroups having maximal embedding dimension with fixed multiplicity and Frobenius number,
\newblock \emph{International Electronic Journal of Algebra}. Published online in 2024. DOI: 10.24330/ieja.1575996

\bibitem{robbiano1986theory}
L. Robbiano,
\newblock On the theory of graded structures,
\newblock \emph{Journal of Symbolic Computation}, vol. 2, no. 2, pp. 139--170, 1986.

\bibitem{rosales2023numerical}
J.~C. Rosales, M.~B. Branco, and M.~A. Traesel,
\newblock Numerical semigroups without consecutive small elements,
\newblock \emph{International Journal of Algebra and Computation}, vol.~33, no.~1, pp.~67--85, 2023.

\bibitem{rosales1999finitely}
J.~C. Rosales, and P.~A. García-Sánchez,
\newblock \emph{Finitely Generated Commutative Monoids},
\newblock Nova Publishers, 1999.

\bibitem{libroRosales}
J.~C. Rosales, and P.~A. García-Sánchez,
\newblock \emph{Numerical Semigroups},
\newblock Developments in Mathematics, vol.~20, Springer, New York, 2009.

\bibitem{rosales2024covariety}
J.~C. Rosales, and M.~A. Moreno-Frías,
\newblock The covariety of saturated numerical semigroups with fixed Frobenius number,
\newblock \emph{Foundations}, vol.~4, no.~2, pp.~249--262, 2024.

\bibitem{rosales2025computational}
J.~C. Rosales, R. Tapia-Ramos, and A. Vigneron-Tenorio,
\newblock A computational approach to the study of finite-complement submonoids of an affine cone,
\newblock \emph{Results in Mathematics}, vol.~80, no.~3, pp.~1--28, 2025.

\bibitem{Sally}
J.~D. Sally,
\newblock On the associated graded ring of a local Cohen-Macaulay ring,
\newblock \emph{Journal of Mathematics of Kyoto University}, vol.~17, no.~1, pp.~19--21.

\bibitem{zariski1&2}
O. Zariski,
\newblock General theory of saturation and saturated local rings: I, II,
\newblock \emph{American Journal of Mathematics}, vol.~93, pp.~573--684, 872--964, 1971.

\bibitem{zariski3}
O. Zariski,
\newblock General theory of saturation and saturated local rings III,
\newblock \emph{American Journal of Mathematics}, vol.~97, pp.~415--502, 1975.

\end{thebibliography}
\end{document}